\documentclass[a4paper]{article}

\usepackage{amsmath,amsfonts,amssymb,amsthm}
\usepackage{listings}
\usepackage{filemod}

\usepackage{graphicx}
\usepackage{tikz}
\usetikzlibrary{shapes,snakes,arrows}
\usepackage{pgf}
\usepackage{float}
\usepackage{xspace}
\usepackage{todonotes}

\newtheorem{theorem}{Theorem}[section]
\newtheorem{proposition}{Proposition}[section]
\newtheorem{lemma}[theorem]{Lemma}

\newtheorem{remark}[theorem]{Remark}

\numberwithin{equation}{section}

\newcommand{\si}{s}
\newcommand{\Si}{S}
\newcommand{\LCP}{LCP\xspace}

\newcommand{\numberset}[1]{\mathbb{#1}}
\newcommand{\R}{\numberset{R}}

\newcommand{\sign}{\mathop{\mathrm{sign}}}

\usepackage{adjustbox}
\usepackage{array}
\newcolumntype{R}[2]{%
	>{\adjustbox{angle=#1,lap=\width-(#2)}\bgroup}%
	l%
	<{\egroup}%
}
\usepackage{graphicx}
\graphicspath{{./images/}{./tikz/}}

\title{Convergence results for some piecewise linear solvers\thanks{The authors wish to thank Lutz Lehmann for his valuable input.}}
\author{
	Manuel Radons \\ Technische Universit\"at Berlin \\
	\and 
	Siegfried M. Rump \\
	Hamburg University of Technology \\
	}

\date{\today}
\begin{document}

\maketitle

\begin{abstract}
Let $A$ be a real $n\times n$ matrix and $z,b\in \mathbb R^n$. The piecewise linear equation system $z-A\vert z\vert = b$ is called an \textit{absolute value equation}. 
%It is equivalent to the general \textit{linear complementarity problem}, and thus NP hard in general. 
We consider two solvers for this problem, one direct, one semi-iterative, and extend their previously known ranges of convergence. 
\end{abstract}

\section{Introduction}

Denote by $\operatorname{M}_n(\mathbb R)$ the space of $n\times n$ real matrices and let 
$A\in\operatorname{M}_n(\mathbb R)$ and $z,b\in \mathbb R^n$. The piecewise linear equation system
\begin{align}\label{AVE_std}
z - A\vert z\vert = b 
\end{align}
is called an \textit{absolute value equation} (AVE) and was first introduced by Rohn in \cite{rohn1989interval}. 
Mangasarian proved its polynomial equivalence to the linear complementarity problem (\LCP) \cite{mangasarian2006absval}.	
In \cite[P. 216-230]{neumaier1990interval} Neumaier authored a detailed survey about the AVEs intimate connection to the research field of \textit{linear interval equations}. Especially closely related system types are equilibrium problems of the form 
\begin{align}\label{brugEq}
Bx+ \max (0, x)= c,
\end{align}
where $B\in \operatorname{M}_n(\mathbb R)$ and $x, c \in \mathbb R^n$. A prominent example is the first hydrodynamic model presented in \cite{brugnano2008iterative}. Using the identity $\max(s, t) = (s+t +\vert s-t \vert)/ 2$, equality \eqref{brugEq} can be reformulated as
\begin{align*} %\label{brugAVE}
Bx + \frac{x+\vert x\vert}{2}=c\quad \Longleftrightarrow \quad (2B+I)x +\vert x \vert =2c,
\end{align*}
and for regular $(2B+I)$ this is equivalent to an AVE \eqref{AVE_std}.

The AVE also has a connection to nonsmooth optimization: 
piecewise affine systems of arbitrary structure may arise as local linearizations of piecewise differentiable objective functions \cite{NewtonPL} or as intermediary problems in the numerical solution of ordinary differential equations with nonsmooth right-hand side \cite{ODE}. Such system can be, with a \textit{one-to-one solution correspondence}, transformed into an AVE \cite[Lem. 6.5]{griewank2014abs}. 

This position at the crossroads of several interesting problem areas has given rise to the development of efficient solvers for the AVE. The latest publications on that matter include approaches by linear programming \cite{mangasarian2014abs} and concave minimization \cite{mangasarian2007concave}, as well as a variety of Newton and fixed point methods (see, e.g., \cite{brugnano2008iterative}, \cite{yuan2012iterative}, \cite{hu2011absval}). In this article we will present and further analyze two solvers for the AVE: the signed Gaussian elimination, which is a direct solver that was developed in \cite{radons2016sge}; and a semi-iterative generalized Newton method developed in \cite{griewank2014abs, streubel2014abs}. In particular, we unify and further extend the known convergence results for both algorithms.

%	Moreover, we will reformulate the frequently discussed XOR-Problem for support vector machines (SVMs) (see, e.g. \cite{bishop2006ml, minsky1969perceptron}) as an AVE and test the algorithms' performance on it. 

\textit{Content and structure of this note:}
In Section \ref{prelimSection} we will assemble the necessary preliminaries from the literature and (re-)prove some auxiliary results. In Section \ref{sec:unify} we will prove a theorem that will allow us to unify and extend the existing correctness, resp. convergence results for the two solvers mentioned above.
In Sections \ref{sgeSection} and \ref{snSection} the solvers are presented and the main results proved.
In Section \ref{sec:compare} we provide an example which demonstrates that despite the results presented in this note, both solvers are not equivalent.

\section{Preliminaries}\label{prelimSection}

We denote by $[n]$ the set $\{1,\dots,n\}$. 
For vectors and matrices \textit{absolute values and comparisons are used entrywise}. Zero vectors and matrices of proper dimension are denoted by~$\mathbf 0$. 
Let $c\in\mathbb R^n$, then we denote by $\operatorname{diag}(c)$ a diagonal matrix in $\operatorname M_n(\R)$ with entries $c_1,\dots, c_n$. %We omit the subscript $n$ and write $\operatorname{diag}(c)$ or $\operatorname{diag}(c_1,\dots, c_n)$ if the dimension is clear from the context.

A \textit{signature matrix} $S$, or, briefly, \textit{a signature}, is a diagonal matrix with entries $+1$ or $-1$, i.e. $\vert S\vert =I$. The set of $n$-dimensional signature matrices is denoted by $\mathcal S_n$. A single diagonal entry of a signature is a sign $s_i$ ($i\in [n]$). Let $z\in \R^n$. We write $\Si_z$ for a signature, where $s_i=1$ if $z_i\ge 0$ and $-1$ else. Clearly, we then have $\Si_zz=\vert z\vert$. Using this convention, we can rewrite (\ref{AVE_std}) as 
\begin{align}\label{AVE}
(I-A\Si_z)z\ =\ b\; .
\end{align}
In this form it becomes apparent that the main difficulty is to determine the proper signature $S$ for $z$. That is, to determine in which of the $2^n$ orthants about the origin $z$ lies. This is NP-hard in general \cite{mangasarian2007absval}.

Denote by $\rho (A)$ the spectral radius of $A$ and let (c.f. \cite[Chap. 5]{rohn1989interval})
\begin{align*}
\rho_0(A)\equiv \max \{\vert \lambda\vert: \lambda\ \text{\textit{real}\ eigenvalue of $A$}\}
\end{align*}
be the \textit{real spectral radius} of $A$. Then its \textit{sign-real spectral radius} is defined as follows (see \cite[Def. 1.1]{rump1997theorems}): 
\begin{align*}
\rho^{\R}(A) \; \equiv \; \max \left \{ \rho_0(S A ) : S \in \mathcal S_n  \right \}. 
\end{align*}
The exponential number of signatures $S$ accounts for the NP-hardness of the computation of $\rho^{\R}(A)$ \cite[Cor. 2.9]{rump1997theorems}. 
For a fixed signature $\bar S$, we have $\{S(\bar S A): S \in \mathcal S_n\}=\{S A: S \in \mathcal S_n\}$. Furthermore, since all $S \in  \mathcal S_n$ are involutive, i.e., $S^{-1}=S$, the spectra of $A$ and $S AS$ are identical. These observations immediately yield the useful identity
\begin{align*}
\rho^{\R}(A)=\rho^{\R}(S_1 A)=\rho^{\R}(AS_2)=\rho^{\R}(S_1AS_2)\qquad \forall\ S_1, S_2 \in \mathcal S_n.
\end{align*} 
%	Recall that a real (or complex) square matrix is called a $P$-matrix if every principal minor is positive \cite[p. 147]{cottle1992lcp}. An LCP has a unique solution for all right hand sides if and only if its system matrix is a $P$-matrix \cite[p. 148, Thm. 3.3.7]{cottle1992lcp}. 
The solvability properties of \eqref{AVE} and the quantity $\rho^{\R}(A)$ are strongly connected (cf. \cite{rump1997theorems}, \cite[p. 220, Thm. 6.1.3-5]{neumaier1990interval}):

\begin{theorem}\label{unique}
	Let $A\in\operatorname M_n(\mathbb R)$. Then the following are equivalent:
	\begin{enumerate}
		\item $ \rho^{\R}(A)\ <\ 1$.
		\item The system $(I-AS_z)z = b$ has a unique solution for all $b\in\mathbb R^n$.
		\item The piecewise linear function $\varphi:\ \mathbb R^n \rightarrow \mathbb R^n,\ z\rightarrow z+A\vert z\vert$ is bijective.
		\item $\operatorname{det}(I-AS)\ >\ 0$ for all $S\in \mathcal S_n$.
	\end{enumerate}
\end{theorem}
We provide a brief assertion of the statements essential to our investigation. For a complete proof of Theorem \ref{unique} we refer to the afore cited references. 
\begin{proof}	
	$(1)\Rightarrow (4):$ $\rho^{\R}(A)\ <\ 1$ implies that the real eigenvalues of all $(I-A\Si), \Si\in\mathcal S_n$, are positive and no complex eigenvalue is $0$\cite{radons2016sge}. 
	
	$(2)\Leftrightarrow (3):$ Clear.
	
	$(3)\Leftrightarrow (4):$ The piecewise linear function is positively homogeneous. 
	Hence, it is bijective if and only it is bijective at the origin. 
	By Clarke's inverse function theorem this is the case if and only if all the matrices $(I-A\Si), \Si\in\mathcal S_n,$, which are the Jacobians of the selection functions of $\varphi$, have the same determinant sign. 
	This sign cannot be negative, because in the described situation all matrices in the polytope $P:=\operatorname{conv}(I-A\Si:\Si\in\mathcal S_n)$ have the same nonzero determinant sign and $P$ contains the identity.
	
	$(4)\Rightarrow (1):$ If it was $\rho^{\R}(A) \ge 1$, we could find a matrix in $P$, as defined in the last step, that was singular.
\end{proof}

There exist various other proofs for the equivalencies listed in Theorem \ref{unique}. See, e.g.,  \cite{rump1997theorems,neumaier1990interval,radons2016sge}. Moreover, note that the sign-real spectral radius is but one facet of the unified Perron-Frobenius theory developed in \cite{rump2002theorems} which extends several key properties of the Perron root of nonnegative real matrices to general real and complex matrices via the concepts of the sign-real and \textit{sign-complex spectral radius}, respectively. A unified expression for these three quantities is derived in \cite[Thm. 2.4]{rump2002theorems}:
$$
\rho^{\mathbb K}(A)\ =\ 
\max\left\{\lvert \lambda\rvert : \lvert Az\rvert = \lvert \lambda z\rvert, \lambda\in\mathbb{K}, z\in\mathbb{K}^n \right\}\ =\
\max_{0\ne x\in\mathbb K^n}\ \min_{x_i\ne 0}\ \left\vert\frac{(Ax)_i}{x_i}\right\vert\,,
$$
where $\mathbb K\in\{\R_+,\R,\mathbb C\}$ and $A\in\operatorname M_n(\mathbb K)$. 

\begin{remark}\label{bounded}
	An important fact is that $\rho^{\R}(A)$ \textit{is bounded by all $p$-norms} \cite[Thm. 2.15]{rump1997theorems}. It affirms that all systems considered in the sequel are \textit{uniquely solvable}. 
\end{remark}

\section{The unifying theorem}\label{sec:unify}
Hereafter, $\operatorname{sign}$ denotes the signum function.
The following simple observation is key to the subsequent discussion: 
\begin{proposition} Let $A\in \operatorname M_n(\mathbb R)$ and $z,b\in\mathbb R^n$ satisfy \eqref{AVE}. If $\lVert A\lVert_{\infty}<1$, then for at least one $i\in [n]$ we have $\operatorname{sign}(z_i)=\operatorname{sign}(b_i)$ . 
\end{proposition}
\begin{proof}
	Let $z_i$ be an entry of $z$ s.t. $\vert z_i\vert \ge \vert z_j\vert\ $ for all $j\in [n]$. If $z_i=0$, then $z=\mathbf 0$ and thus $b\equiv z-A\vert z\vert =\mathbf{0}$, and the statement holds trivially. If $\vert z_i\vert >0$, then $\left| e_i^\intercal A\vert z\vert \right| < \vert z_i\vert$, due to the norm constraint on $A$. Thus, $b_i = z_i - e_i^\intercal A\vert z\vert$ will adopt the sign of $z_i$. 
\end{proof}

We do not know though, for which indices the signs coincide. The theorem below states restrictions on $A$ which guarantee the coincidence of the signs of $z_i$ and $b_i$ for all $i\in [n]$ where $\vert b_i\vert = \|b\|_\infty$ and thus provide the basis for the convergence proofs in Sections \ref{sgeSection} and \ref{snSection}. For $b\in \mathbb R^n$ we set 
$$
\mathcal I^b_{\max}\ \equiv \left\{\, i\in[n] \ : \ \vert b_i\vert = \| b\|_\infty\, \right\}\; ,
$$ 
and define 
$$\operatorname{Neq}(A,b,z)\ \equiv\ \{ i\in \mathcal I^b_{\max}\ :\ \operatorname{sign}(b_i) \ne \operatorname{sign}(z_i) \}\; .$$

\begin{theorem}\label{citeSGE}
	Let $A\in\operatorname M_n(\mathbb R)$ and $b, z\in \mathbb R^n$ such that \eqref{AVE} is satisfied. Then the we have 
	$$\operatorname{Neq}(A,b,z)\ =\ \emptyset$$
	if either of the following conditions is satisfied.
	\begin{enumerate}
		\item $\lVert A\lVert_{\infty}\ <\ \frac 12$\,.
		\item $A$ is irreducible, and $\lVert A\lVert_{\infty}\ \le\ \frac 12$\,.
		\item $A$ is strictly diagonally dominant and $\Vert A\Vert_\infty \le \frac 23$\,.
		\item $\lvert A\rvert$ is tridiagonal, symmetric, $\Vert A\Vert_\infty <1$, and $n\ge 2$\,.
	\end{enumerate}
\end{theorem}
The first three points are cited from \cite[Thm. 3.1]{radons2016sge}. We will prove the fourth point and reprove the first two by somewhat more elegant means than in the latter reference. This includes a new proof for the following lemma.

\begin{lemma}(\cite[Lem. 3.2]{radons2016sge}) \label{irred}
	Let $A\in M_n(\R)$
	with $\|A\|_\infty < \frac12$ or irreducible with norm $\|A\|_\infty\le \frac12$.
	Then the inverse of $B=I-A$ is strictly diagonally dominant
	and has a positive diagonal.
\end{lemma}

\begin{proof}

	As $\|A\|_\infty\le\frac12<1$, the inverse of $(I-A)$ exists and can
	be expressed via the Neumann series
	$$
	(I-A)^{-1}=I+\sum_{k=1}^\infty A^k=I+A(I-A)^{-1}
	\text{ with }
	\|A(I-A)^{-1}\|_\infty\le \frac{\|A\|_\infty}{1-\|A\|_\infty}\le 1.
	$$
	This already proves strict diagonal dominance for $\|A\|_\infty<\frac12$.
	Moreover, due to $\lVert A\rVert_\infty<1$, we have $\rho(A')<1$ for any principal submatrix $A'$ of $A$ (including the case $A'=A$). 
	This implies that the real part of all eigenvalues of $I-A'$ is positive. 
	Consequently, all real eigenvalues of $I-A'$ are positive and no complex eigenvalue is $0$. 
	Since the complex eigenvalues of real marices appear in conjugate pairs, whose product is positive as well, we get $\det(I-A')>0$.
	The positivity of the diagonal of $(I-A)^{-1}$ now follows from Cramer's rule.
	
	To further explore the diagonal dominance of a matrix sum $I+M+R$, 
	where we will use
	$M=A^m$ and $R=\sum_{k\ne m}A^k$, we
	bound the gap in the inequality below as
	\begin{align}\label{eq:digdomsep}
	|1+M_{ii}+R_{ii}|-\sum_{j\ne i}|M_{ij}+R_{ij}|
	&\ge |1+M_{ii}|-|R_{ii}|-\sum_{j\ne i}(|M_{ij}|+|R_{ij}|)
	\notag\\
	&=|1+M_{ii}|+|M_{ii}|-\sum_{j=1}^n(|M_{ij}|+|R_{ij}|)\; .
	\end{align}
	Thus we get strict diagonal dominance in row $i$ both for 
	$\|M\|_\infty+\|R\|_\infty<1$
	and in the case of $\|M\|_\infty+\|R\|_\infty=1$ and $M_{ii}>0$,  where 
	the partition
	$A(I-A)^{-1}=M+R$ can be chosen differently for every $i=1,\dots,n$.

	If  $\rho(A)<\frac12$, then
	$(2A)^k$ converges toward zero, so that there is some $K$ with 
	$\|(2A)^K\|\le\frac12 \ (\Leftrightarrow \|A^K\|_\infty \leq 2^{-(K+1)})$ and thus
	$$
	\sum_{k=1}^\infty \|A^k\|_\infty\
	\le \
	\frac{\|A\|_\infty}{1-\|A\|_\infty}\,\frac{1-\|A\|_\infty^K}{1-\|A^K\|_\infty}\
	\le\
	\frac{1-2^{-K}}{1-2^{-(K+1)}}\
	<\ 1\; ,
	$$
	ensuring strict diagonal dominance of $(I-A)^{-1}$.
	The left inequality follows from 
	$$
	\sum_{k=1}^\infty \|A^k\|_\infty\ \le\ \sum_{k=1}^\infty \|A\|_\infty^k\ =\ \|A\|_\infty \sum_{k=0}^\infty \|A\|_\infty^k\ =\ \frac{\|A\|_\infty}{1-\|A\|_\infty}\,.
	$$
	The right one follows from 
	$$
	\frac{\|A\|_\infty(1-\|A\|_\infty^K)}{1-\|A\|_\infty}\ =\ \sum_{k=1}^K\|A\|_\infty^k\ < \ 1-2^{-K}
	$$
	and the inequality $1 - \|A^K\|_\infty \geq 1 - 2^{-(K+1)}$, which holds by hypothesis.
		
	In the case $\rho(A)=\frac12$ the assumptions of the Wielandt theorem \cite{meyer2000matrix, wielandt1950matrix} are satisfied,
	$\rho(A)=\rho(|A|)=\|A\|_\infty$, such that there
	is a sign $s$ and a signature matrix $\Si={\rm diag}(\si_1,\dots,\si_n)$ 
	with
	$|s|=|\si_i|=1$, $i=1,\dots,n$, so that
	$
	A=s\,T^{-1}\,|A|\,T
	$
	and thus for the  powers of $A$
	$$
	\left|A^k\right|=\left|\,s^k T^{-1}\,|A|^k\,T\,\right|=|A|^k.
	$$
	The diagonal elements of $|A|^k$ are sums of products over $k$-cycles of 
	non-negative elements.
	Since $|A|$ is irreducible there is at least one $k_i$-cycle, $k_i\in [n]$, for each
	diagonal element at position $i$. Thus we find 
	$(|A|^{k_i})_{ii}=|(A^k)_{ii}|>0$.
	For the square of that power we note that the diagonal element satisfies 
	the identity
	\begin{equation}
	\left|\sum_{j=1}^n(A^{k_i})_{ij}(A^{k_i})_{ji}\right|
	=\left(\left|A^{2k_i}\right|\right)_{ii}
	=\left(\left|A^{k_i}\right|^2\right)_{ii}
	=\sum_{j=1}^n(|A^{k_i}|)_{ij}(|A^{k_i}|)_{ji} \; .
	\end{equation}
	
	By the triangle inequality, the identity of the leftmost and rightmost 
	terms is only possible if
	all the terms in the sum on the left have the same sign. As 
	$(A^{k_i})_{ii}^2>0$, all those terms 
	are positive and consequently $(A^{2k_i})_{ii}>0$, which proves diagonal dominance in 
	row $i$
	by setting $M=A^{2k_i}$ and $R=\sum_{m\ne 2k_i}A^m$ in the separation 
	inequality \eqref{eq:digdomsep}.
	This can be done for any index thus proving overall diagonal dominance 
	of $(I-A)^{-1}$.
\end{proof}

\begin{proof}({\bf Theorem \ref{citeSGE}}) 		
	$(1)$ and $(2)$:  Let $A\in M_n(\R)$
	with $\|A\|_\infty < \frac12$ or irreducible with norm $\|A\|_\infty\le \frac12$. Moreover, assume set $\Si_z\equiv\Si$. Then $(I-AS)^{-1}$ is strictly diagonally dominant by Lemma \ref{irred} since both irreducibility and the norm constraint are invariant under multiplications of $A$ by a signature matrix. Hence, $z_i=e_i^\intercal(I-AS)^{-1}b$ will adopt the sign of $b_i$ for all $i\in\mathcal I^b_{\max}$.  
	
	$(3)$: See \cite{radons2016sge}.
	
	$(4)$: The proof proceeds by induction. The $(2\times 2)$-case can be verified by direct computation. (We note that it is the only part of the proof that makes use of the symmetry of $A$.) Now assume the statement of the theorem holds for an $N\ge 2$, but the tuple $(A,z,b)$ contradicts it in dimension $N+1$. 
	Since $ \lVert A\lVert_{\infty} <1$, it is $\operatorname{sign}(z_i)=\operatorname{sign}(b_i)$ for all $i\in[n]$, if $\vert z_1\vert=\dots =\vert z_n\vert$. Thus we will assume that not all entries of $z$ have the same absolute value.

	For all entries $z_j$ of $z$ whose abolute value is maximal in $z$, we then have $\sum_j \vert a_{ij} z_j\vert < \vert z_i\vert$ and hence $\operatorname{sign}(b_i) = \operatorname{sign}(z_i)$.
	Consequently, if there existed a tuple $(A,z,b)$ which violated the claim of the theorem, for any $i \in \operatorname{Neq}(A,b,z)$ the absolute value of $z_i$ would not be maximal in $z$. 
	
	As $N+1\geq 3$, we can assume without loss of generality that $|z_N|=\lVert z\rVert_\infty$, while the first row holds the contradiction, so that we have $|b_1|=\lVert b_1\rVert_\infty$ and $\sign(b_1)\neq\sign(z_1)$.
	Then there exists a scalar $\zeta \in \left[0, 1 \right]$ such that 
	$$
	\zeta \cdot \vert z_N\vert = \vert z_{N+1}\vert\quad \Longrightarrow\quad \zeta\cdot a_{N,N+1} \cdot \vert z_N\vert = a_{N,N+1}\cdot \vert z_{N+1}\vert.
	$$ 
	Let $A\in \operatorname M_{N+1}(\R)$ and denote by $A_{N+1,N+1}$ the matrix in $M_N(\R)$ that is derived from $A$ by removing its $(N+1)$-th row and column. 
	Then $$\bar A\ \equiv\ A_{N+1,N+1} + \operatorname{diag}(0,\dots, 0, \zeta\cdot a_{N,N+1})$$ is still tridiagonal with $\lVert A\lVert_{\infty}<1$ and $\vert A\vert$ symmetric. 
	Further, for $\bar z\equiv (z_1,\dots, z_N)^\intercal$ and $\bar b\equiv (b_1,\dots, b_N)^\intercal$ we have $$\bar z-\bar A\vert \bar z\vert\ =\ \bar b.$$ 
	Hence, the tuple $(\bar A,\bar z, \bar b)$ contradicts the induction hypothesis for dimension $N$, as we still have $1\in\operatorname{Neq}(\bar A,\bar z, \bar b)$. 
\end{proof}

\section{Signed Gaussian elimination}\label{sgeSection}

If one is sure of the sign $\si_k$ of $z_k$ one can remove this variable from the right hand side of the AVE.
Let $A_{\ast k}$ denote the $k$-th column $Ae_k$ and $A_{j\ast}$ the $j$-th row $e_j^\intercal A$.
Then the removal of the variable is reflected in the formula 
$$
(I-A_{\ast k}e_k^\intercal\si_k)z = b + (A-A_{\ast k}e_k^\intercal)|z|.
$$
The inverses of rank-1-modifications are well-known to be (see, e.g. \cite{bartlett1951smw})
\begin{align*}
(I-uv^\intercal)^{-1}=I+\frac1{1-v^\intercal u}uv^\intercal\,.
\end{align*}
Thus it is
easy to remove the matrix factor on the left side. We then have
\begin{align}\label{reduction1}
z&=\bar b + \bar A|z|\; ,
\end{align}

\vspace{-.3cm}

\noindent where
\begin{align*}
\bar b = b + \frac1{1-A_{kk}\si_k}\si_kA_{\ast k}\,b_k\quad\text{and}\quad \bar A = A_{red}+ \frac1{1-A_{kk}\si_k}\si_kA_{\ast k}\,(A_{red})_{k\ast}\; ,
\end{align*}
with
\begin{align*}
A_{re d}=A-A_{\ast k}e_k^\intercal=A(I-e_ke_k^\intercal)\; .%\label{reduction3}
\end{align*}
In {\sc Python} this can be achieved as:
\vspace{.2cm}
\renewcommand{\lstlistingname}{Function}
\begin{lstlisting}[caption=Sign Controlled Elimination Step, label=sgeStep]
def elim(A,b,k,sigk):
sk = A[:,k]*sigk; 
A[:,k] = 0;
sk = sk/(1-sk[k])
b = b + sk*b[k];
A = A + sk*A[k,:];
\end{lstlisting}
\vspace{.3cm}
This procedure corresponds to one step of Gaussian elimination. Now let $J\subseteq [n]$ be an index set and define   
$$\operatorname{J\_b}\ \equiv \{ i\in J\ : \ \vert b_i\vert \ge \vert  b_j\vert\ \forall j\in J \}\; .$$
Using this convention we can give the pseudocode of a slight modification of the algorithm that was introduced as \textit{signed Gaussian elimination} (SGE) in \cite{radons2016sge}: 
%\vspace{.2cm}

\renewcommand{\lstlistingname}{Algorithm}
\begin{lstlisting}[caption=Signed Gaussian Elimination, label=sgePseudo]
sge(A,b):
set J = [n];
while (#J > 1) do:
determine bmaxJ; 
forall k in J_b set sigk = sig(bk);
forall k in J_b elim(A,b,k,sigk);
J = J\ bmaxJ;
endwhile
perform reverse substitution for (I-A)z = b;
return z
\end{lstlisting}

\begin{theorem}\label{sgeCorr}
	Let $A\in\operatorname M_n(\R)$ and $z, b\in \R^n$ such that \eqref{AVE_std} is satisfied. If $A$ conforms to any of the conditions listed in Theorem \ref{citeSGE}, then the signed Gaussian elimination computes the unique solution of the AVE \eqref{AVE_std} correctly.
\end{theorem}
\begin{proof}It was already noted that criteria $(1)$-$(4)$ imply the unique solvability of the AVE (Remark \ref{bounded}). We may thus focus on proving the correctness of the algorithm: 
	
	Theorem \ref{citeSGE} ensures the correctness of the sign-picks. The conditions listed in Theorem \ref{citeSGE} are clearly invariant under the (sign controlled) elimination step. Hence, the argument applies recursively down to the scalar level.
\end{proof}
For dense $A$ the SGE has a cubic computational cost. For $A$ with band structure it was shown in \cite{radons2016sge} that the computation has the asymptotic cost of sorting $n$ floating point numbers. Moreover, note that the SGE is numerically stable, since $I-AS$ is strictly diagonally dominant if $\Vert A\Vert_\infty<1$.

For counterexamples which demonstrate the sharpness of the conditions $(1)$-$(3)$ in Theorem \ref{citeSGE} with respect to the SGE's correctness, see \cite{radons2016sge}. 
Concerning condition $(4)$, let $A\equiv (1+\varepsilon)I$, where $\varepsilon>0$ is arbitrarily small.
Then we have $\lVert A\rVert_\infty = 1+\varepsilon$ and the SGE will pick the wrong first sign for arbitrary right hand sides $b$.

\section{Full step Newton method}\label{snSection}
In this section we analyze the full step Newton method (FN) which is defined by the recursion 
%note we will focus on the presentation of the generalized Newton iteration

\vspace{-.6cm}

\begin{align}\label{newtonz}
z^{k+1}\ =\ (I-AS_k)^{-1}b\; ,
\end{align}

\vspace{-.3cm}

\noindent where $\Si_k\equiv \Si_{z^k}$. The iteration has the terminating criterion

\vspace{-.5cm}

\begin{align*}
z^k=z^{k+1}\; .
\end{align*}

\vspace{-.3cm}

\noindent It was developed in \cite{griewank2014abs} and is equivalent to the semi-iterative solver for the equilibrium problem \eqref{brugEq} developed in \cite{brugnano2008iterative}. 
A first, albeit rather restrictive, convergence result is \cite[Prop. 7.2]{griewank2014abs}:

\begin{proposition}
	If $\|A\|_p < 1/3$ for any $p$-norm, then the iteration \eqref{newtonz} converges for all $b$ 
	in finitely many iterations from any $z_0$ to the unique solution of \eqref{AVE}. Moreover,  the p-norms of both $z_i-z$ as well as $(I-AS_i)z_{i+1}-b$ are monotonically reduced.
\end{proposition}

Moreover, in \cite[Prop. 7]{griewank2014abs} convergence was proved for the first two restrictions on $A$ in Theorem \ref{citeSGE}. The following extends this result to the criteria in Theorem \ref{citeSGE}.3-4. 

\begin{theorem}\label{snConv}
	Let $A\in\operatorname M_n(\R)$ and $z,b\in \R^n$ such that \eqref{AVE_std} is satisfied. If $A$ conforms to any of the conditions listed in Theorem \ref{citeSGE}, then for any initial vector $z^0\in\R^n$ the full step Newton method \eqref{newtonz} computes the unique solution of the AVE \eqref{AVE_std} correctly in at most $n+1$ iterations.
\end{theorem}
\begin{proof} Note that all conditions listed in Theorem \ref{citeSGE} are invariant under scalings of $A$ by a signature matrix.
	Now assume that $z$ satisfies the equality 
	
	\vspace{-.3cm}
	
	$$z-Az\ =\ b$$

	\vspace{-.2cm}
	
	\noindent and set $S\equiv \Si_z$. Then, since $SS=I$, we have 
	
	\vspace{-.3cm}
	
	$$b\ =\ z-ASS z\ \equiv z-A'\vert z\vert\; ,$$ 
	
	\vspace{-.1cm}
	
	\noindent and $A'$ is still strictly diagonally dominant with $\Vert A\Vert_\infty<1$. This implies $\operatorname{Neq}(A',b,z)$ is empty, by Theorem \ref{citeSGE}. Hence, we have 
	
	\vspace{-.6cm}
	
	$$s'_{i}\ =\ 1\qquad \text{if}\quad b_i\ge 0\qquad \text{and}\qquad s'_{i}\ =\ -1\qquad \text{else, for all}\ i\in b_{max}\; .$$
	
	\vspace{-.2cm}
	
	\noindent That is, the signs with index in $\mathcal I^b_{\max}$ are fixed throughout all iterations. Now assume $i\in\mathcal I^b_{\max}$. Then for all $k\ge 1$ we will have $\operatorname{sign}(z_i^k)=\operatorname{sign}(b_i)$. This allows us to rewrite the $i$-th equation in \eqref{AVE_std} and express the $z^k_i$ as a linear combination of the other $z^k_j$ by the transformations of $A$ and $b$ to $\bar A$ and $\bar c$, respectively, as they were defined in \eqref{reduction1}. This corresponds to one step of Gaussian elimination. As mentioned in the proof of Theorem \ref{sgeCorr}, all restrictions listed in Theorem \ref{citeSGE}.1-4 are invariant under the latter operation, which implies that the argument applies recursively and all signs of $z$ are fixed correctly in at most $n+1$ iterations.  Again, we remark that the conditions in Theorem \ref{citeSGE}.1-4 imply the uniqueness of the solution at which we arrive via the afore described procedure. 
\end{proof}

\section{Comparison of both solvers}\label{sec:compare}
Despite their proved range of correctness, resp. convergence being the same, both solvers are not equivalent and neither solver has a stricly larger range than the other:
Let
	$$A\equiv \begin{bmatrix} \frac{\varepsilon}{2}& \frac{1+\varepsilon}{2}  \\ 0 & \frac 12 \end{bmatrix} \qquad \text{and} \qquad    z\equiv\begin{bmatrix} \frac{\varepsilon}{2}  \\ 1 \end{bmatrix}\;, $$
	where $\varepsilon>0$ is arbitrarily small.
	Then, for $b\equiv z-A\vert z\vert$ we have $b=(-\frac{2+\varepsilon^2}{4}, \frac 12)^\intercal$. And clearly $\vert b_1\vert >\vert b_2\vert$, but  $\operatorname{sign}(b_1)\ne \operatorname{sign}(z_1)$. It was shown in \cite[Prop.5.2]{radons2016sge} that for $A$ and $b$ as described, the SGE is lead astray -- but we have $\lVert A\rVert_\infty = \frac 12 + \varepsilon$. An elementary calculation shows that for $n\le 2$ we have convergence of the FN method if $\Vert A\Vert_\infty <1$ \cite{radons2016master}. 	
	In \cite[Sec. 7]{griewank2014abs} it was shown that for a system with 
	$$A\equiv \begin{bmatrix} 0 & 0 & a \\ a & 0 & 0 \\ 0 & a & 0 \end{bmatrix} \qquad \text{and} \qquad    b\equiv\begin{bmatrix} 1  \\ 1 \\ 1 \end{bmatrix}\;, $$
	where $a=5/8$, the FN method cycles for all starting signatures that contain both positive and negative signs.
	It is sraightforward to show that the SGE solves the corresponding AVE, but we refrain from performing this exercise.
	
	\vspace{-.5cm}
	
\bibliographystyle{alpha}
\bibliography{references}

\end{document}